\documentclass[reqno]{amsart}
\usepackage[scale=0.75, centering, headheight=14pt]{geometry}
\usepackage[latin1]{inputenc}
\usepackage[T1]{fontenc}
\usepackage{lmodern}
\usepackage[english]{babel}

\usepackage{tikz}
\usepackage{bbm}
\usepackage{amsmath,amssymb,amsfonts,amsthm}
\usepackage{mathtools,accents}
\usepackage{mathrsfs}
\usepackage{xfrac}
\usepackage{array} 
\usepackage{aliascnt}
\usepackage{booktabs} 
\usepackage{array} 

\usepackage{verbatim} 
\usepackage{subfig} 

\usepackage{mathrsfs, dsfont}
\usepackage{amssymb}
\usepackage{amsthm}
\usepackage{amsmath,amsfonts,amssymb,esint}
\usepackage{graphics,color}
\usepackage{enumerate}
\usepackage{mathtools,centernot}

\usepackage{microtype}
\usepackage{paralist} 
\usepackage{cases}
\usepackage[initials]{amsrefs}
\allowdisplaybreaks

\usepackage{braket}
\usepackage{bm}

\usepackage[citecolor=blue,colorlinks]{hyperref}
\addto\extrasenglish{}

\usepackage{enumerate}
\usepackage{xcolor}
\usepackage{aliascnt}

\makeatletter
\def\newaliasedtheorem#1[#2]#3{
  \newaliascnt{#1@alt}{#2}
  \newtheorem{#1}[#1@alt]{#3}
  \expandafter\newcommand\csname #1@altname\endcsname{#3}
}
\makeatother

\theoremstyle{plain}
\newtheorem{theorem}{Theorem}[section]
\newtheorem{theoremint}{Theorem}[]

\newaliasedtheorem{lemma}[theorem]{Lemma}
\newaliasedtheorem{prop}[theorem]{Proposition}
\newaliasedtheorem{claim}[theorem]{Claim}
\newaliasedtheorem{corollary}[theorem]{Corollary}
\theoremstyle{remark}

\theoremstyle{definition}
\newaliasedtheorem{definition}[theorem]{Definition}
\newaliasedtheorem{example}[theorem]{Example}
\newtheorem{OQ}[theoremint]{Open Question}

\theoremstyle{remark}
\newaliasedtheorem{remark}[theorem]{Remark}

\numberwithin{equation}{section}

\def\eps{\varepsilon}
\def\tr{\text{tr}}
\def\R{\mathbb R}
\def\N{{\mathbb N}}

\def\E{\mathds E}

\DeclareMathOperator{\diver}{div}
\DeclareMathOperator{\curl}{curl}
\DeclareMathOperator{\spt}{spt}
\DeclareMathOperator{\diag}{diag}
\DeclareMathOperator{\supp}{supp}
\DeclareMathOperator{\id}{id}
\DeclareMathOperator{\rank}{rank}

\DeclareMathOperator{\out}{out}
\DeclareMathOperator{\stat}{stat}
\DeclareMathOperator{\loc}{loc}

\newcommand{\KK}{K^g_{\stat}}

\DeclareMathOperator{\Lip}{Lip}
\DeclareMathOperator{\cof}{cof}

\DeclareMathOperator{\dist}{dist}
\DeclareMathOperator{\Det}{det}
\DeclareMathOperator{\dv}{div}

\title{Critical points of degenerate polyconvex energies}

\author{Riccardo Tione}
\address{R.T.: Max Planck Institute for Mathematics in the Sciences, Inselstrasse 22, 04103 Leipzig, Germany,}
\email{riccardo.tione@mis.mpg.de}

\begin{document}

\maketitle

\begin{abstract}
We study critical and stationary, i.e. critical with respect to both inner and outer variations, points of polyconvex functionals of the form $f(X) = g(\Det(X))$, for $X \in \R^{2\times 2}$. In particular, we show that critical points $u \in \Lip(\Omega,\R^2)$ with $\det(Du) \neq 0$ a.e. have locally constant determinant except in a relatively closed set of measure zero, and that stationary points have constant determinant almost everywhere. This is deduced from a more general result concerning solutions $u \in \Lip(\Omega,\R^n)$, $\Omega \subset \R^n$ to the \emph{linearized} problem $\curl(\beta Du) = 0$. We also present some generalization of the original result to higher dimensions and assuming further regularity on solutions $u$. Finally, we show that the differential inclusion associated to stationarity with respect to polyconvex energies as above is rigid.
\end{abstract}
\par
\medskip\noindent
{\sc MSC (2020): 35B38, 35B65, 35D30, 35G50, 35J70.
\par
}
\section{Introduction}
Consider the \emph{polyconvex} energy defined by
\begin{equation}\label{EN}
\mathds{E}(u) \doteq \int_{\Omega}f(Du)dx, \text{ for }u \in \Lip(\Omega,\R^n), \text{ with } f(X) = g(\det X),
\end{equation}
where $\Omega$ is an open and convex subset of $\R^n$ and $g :\R\to\R$ is a $C^2$ and uniformly convex function. This type of functionals has been extensively considered in the last years, mainly in connection with the theory of gradient flows, see for instance \cite{ESG,CL}. Critical points for (\emph{outer}) variations of the form
\[
\frac{d}{d\eps}|_{\eps = 0}\E(u + \eps v) = 0,
\]
satisfy the Euler-Lagrange equation:
\begin{equation}\label{EL}
\dv(Df(Du)) = \diver(g'(\det(Du))\cof^T(Du)) = 0.
\end{equation}
We refer the reader to Subsection \ref{Not} for the notation used in this paper. Every Lipschitz map with constant determinant solves system \eqref{EL}, which is a consequence of the fact that $\diver(\cof^T(Du)) = 0$ for every Lipschitz $u$, see \cite[Ch. 8, Thm. 2]{EVA}. If we assume $u$ smooth, this is the only kind of solution. To see this, let us assume for simplicity that $\Det(Du) \neq 0$ everywhere, even though this is not necessary. Since $\diver(\cof^T(Du)) = 0$, we rewrite \eqref{EL} as
\[
\cof^T(Du)D(g'(\det(Du))) = 0.
\]
If $Du$ is invertible at every point, one obtains that $g'(\det(Du))$ is constant. If one considers, as we will, strictly convex functions $g$, one further obtains that $\det(Du)$ is constant. This computation holds for $C^2$ solutions, but it is unclear if it might also hold for Lipschitz or even $C^1$ solutions. One of the main results of this paper shows that analogous statements are true for Lipschitz solutions. First of all, let us fix some hypotheses on $g$. We say that $g$ satisfies the set of hypothesis \eqref{HP} if
\begin{equation}\label{HP}\tag{HP}
g \in C^2(\R),\quad g''(t) > 0, \forall t \in \R,\quad\;g(0) = g'(0) = 0.
\end{equation}
Notice that the simplification $g(0) = g'(0) = 0$ can always be made, and thus the only real structural assumption is on the convexity of $g$. We will show that:

\begin{theoremint}\label{thm:poly}
Let $\Omega \subset \R^2$ be open and connected. Let $u \in \Lip(\Omega,\R^2)$ and $g$ satisfy \eqref{HP}. Suppose that \eqref{EL} holds in the weak sense and that
\begin{equation}\label{notzeroint}
\det(Du) \neq 0, \text{ a.e. in }\Omega.
\end{equation}
Then, $\det(Du)$ is \emph{essentially} locally constant, in the sense that there exists an open set $\Omega_0$ with $|\Omega\setminus\Omega_0| = 0$ such that, on $\Omega_0$, $x\mapsto \det(Du)$ is locally constant. If, moreover, there exists $\delta > 0$ such that $\det^2(Du)\ge \delta$, then $\det(Du)$ is constant in $\Omega$.
\end{theoremint}
This theorem can be read as an interior regularity result for certain critical points of the energy $\mathds{E}$. One cannot expect better regularity properties of a solution of \eqref{EL} in general, for instance higher differentiability of $u$ or continuity of its partial derivatives, since \eqref{EL} does not carry any further information once we have proved the constancy of $x\mapsto \Det(Du)(x)$. This is due to the lack of a uniformly convex term in the definition of the integrand $f$. For instance, one might study regularity properties of critical points of the energy $\mathds{E}$ associated to the integrand
\begin{equation}\label{polycoer}
f(X) = \frac{|X|^2}{2} + g(\det(X)).
\end{equation}
No general positive regularity results are known for functional of this form, except for the case of minimizers, see \cite{EVA,AFR,KRI}. On the other hand, for more general functionals of the form
\[
f(X) = \frac{|X|^2}{2} + g(X,\det(X)),
\]
it is shown by L. Sz{\'{e}}kelyhidi in \cite{LSP} that no partial regularity result can hold. We also refer the interested reader to \cite{SMVS} for the proof in the case of quasiconvex functionals. These results make the question of partial regularity for critical points of polyconvex fuctionals very delicate. We will also be interested in considering stationary points for the energy $\E$ of \eqref{EN}. We say that $u \in \Lip(\Omega,\R^n)$ is a stationary point (for $\mathds{E}$) if $u$ solves in the weak sense the systems
\begin{equation}\label{vargr}
\dv(Df(Du)) = 0 \;\quad\text{and}\quad \dv(Du^TDf(Du) - f(Du)\id) = 0
\end{equation}
The latter is the Euler-Lagrange equation arising from \emph{inner} (or domain) variations, which are defined as follows. Given a vector field $\Phi\in C^1_c (\Omega, \R^n)$, we let $X_\varepsilon$ be its flow\footnote{Namely $X_\varepsilon(x) = \gamma_x(\varepsilon)$, where $\gamma_x$ is the solution of the ODE
$\gamma'(t) = \Phi(\gamma(t))$ subject to the initial condition $\gamma(0) = x$.}. The one-parameter family of functions $u_\varepsilon = u \circ X_\varepsilon$ will be called an {\em inner variation}. A map $u \in \Lip(\Omega,\R^n)$ is \emph{critical for inner variations} for $\mathds{E}$ if
\[
\left.\frac{d}{d\varepsilon}\right|_{\varepsilon = 0} \mathds{E}(u_\varepsilon) = 0,\qquad \forall \Phi\in C^1_c(\Omega,\R^n)\, ,
\]
which is equivalent to $\dv(Du^TDf(Du) - f(Du)\id) = 0$ in the weak sense. The situation for stationary points to certain polyconvex functionals seems to be more rigid than the one for critical points, see \cite{DLDPKT,TR,JTR, AR,TAH}. This is the case also in the problem considered in this paper. In fact, we have:
\begin{theoremint}\label{cor:poly}
Let $\Omega \subset \R^2$ be open and connected and let $u\in \Lip(\Omega,\R^2)$ be a stationary point of the energy \eqref{EN}. Then, $x \mapsto \det(Du)(x)$ is constant.
\end{theoremint}
Theorems \ref{thm:poly}-\ref{cor:poly} will be deduced as corollaries of the following more general result:
\begin{theoremint}\label{lem:constint}
Let $\Omega\subset \R^m$ be open and connected and let $n \ge m\ge 2$. Let $u\in \Lip(\Omega,\R^n)$ and $\beta \in L^\infty(\Omega)$ satisfy
\[
\curl(\beta Du) = 0
\]
in the weak sense. Suppose that there exists an $m \times m$ minor $M$ such that
\[
\beta\Det(M(Du)) > 0, \text{ a.e. in }\Omega.
\]
Then, $\beta$ is \emph{essentially} locally constant, in the sense that there exists an open set $\Omega_0$ with $|\Omega\setminus\Omega_0| = 0$ such that, on $\Omega_0$, $\beta$ is locally constant. Moreover, if there exists $\delta > 0$ such that
\[
\beta\Det(M(Du)) \ge \delta, \text{ a.e. in } \Omega,
\]
 then, $\beta$ is constant. 
\end{theoremint}
The essential ingredients for proving Theorem \ref{lem:constint} are \cite[Theorem 3.1]{FON} and topological properties of quasiregular mappings. These will be recalled in Section \ref{subs:inv}. The results stated above will occupy most of Section \ref{s:exact}. In that Section, we will also show Theorem \ref{thm:poly} in every dimension under a stronger assumption than \eqref{notzeroint}, and we will show the same theorem assuming $u \in \Lip\cap W^{2,1}(\Omega,\R^n)$, using a result of \cite{DETREM}.
\\
\\
Finally, in Section \ref{s:approx}, we will consider \emph{approximate} solutions to \eqref{vargr}. In \cite[Chapter 5]{KMS}, which was also the starting point of our interest in the problem, B. Kirchheim, S. M\"uller and V. \v Sver\'ak study the case $n = 2$ and $g(t) = t^2$, i.e.
\[
f(X) = \det(X)^{2}.
\]
They do so by showing properties of the set
\begin{equation}\label{kms}
K \doteq
\left\{A \in \R^{4\times 2}:
A =
\left(
\begin{array}{c}
X\\
\Det(X)X
\end{array}
\right)
\right\}.
\end{equation}
The set $K$ is related to \eqref{EL} in the following way. First of all, since $\Omega$ is two dimensional, we notice that \eqref{EL} is equivalent to
\[
\curl(g'(\det(Du))\cof^T(Du)J) = 0,
\]
where
\[
J=
\left(
\begin{array}{cc}
0& -1\\
1 & 0
\end{array}
\right).
\]
Moreover, a direct computation shows that for every $X \in \R^{2\times 2}$, $\cof^T(X)J = JX$. Therefore,
\begin{equation}\label{equivalences}
\dv(g'(\det(Du))\cof^T(Du)) \Leftrightarrow \curl(g'(\det(Du))\cof^T(Du)J)=0  \Leftrightarrow \curl(g'(\det(Du))Du) = 0.
\end{equation}
If we work locally and assume without loss of generality that $\Omega$ is convex, we can use Poincar\'e's Lemma to deduce that $u$ satisfies \eqref{EL} if and only if there exists $v \in \Lip(\Omega,\R^2)$ such that
\[
Dv = g'(\det(Du))Du,
\]
or in other words if and only if $w:\Omega\to\R^4$ defined as $w\doteq \left(\begin{array}{c}u\\v\end{array}\right)$ solves
\begin{equation}\label{diffinc}
Dw(x)\in K^g\doteq\left\{A \in \R^{4\times 2}:
A =
\left(
\begin{array}{c}
X\\
g'(\Det(X))X
\end{array}
\right)
\right\}, \text{ for a.e. }x\in \Omega.
\end{equation}
When $g(t) = t^2$, we will simply denote $K^g$ with $K$. Analogously, we may translate system \eqref{vargr} for $f(X) = g(\Det(X))$ into the differential inclusion defined by
\[
K^g_{\stat} = \left\{A \in \R^{6\times 2}:
A =
\left(
\begin{array}{c}
X\\
g'(\Det(X))X\\
(g'(\Det(X))\Det(X) - g(\Det(X)))J
\end{array}
\right)
\right\}.
\]
The fact that equation \eqref{EL} admits the equivalent formulations \eqref{equivalences}-\eqref{diffinc} will be fundamental for the proofs of this paper, and will be used many times. One of the main results of \cite[Chapter 5]{KMS} is the proof that the rank-one convex hull of $K$ is trivial, $K^{rc} = K$. Let us postpone to Subsection \ref{APPSOL} the definition of the hulls and Young measure. They leave as an open question \cite[Question 10]{KMS} the following:
\begin{OQ}\label{open}
Is the quasiconvex hull of $K$ related to $\Det(X)^2$ trivial?
\end{OQ}
In other words, do we have that $K^{qc} = K$? We will not be able to answer this question in this generality, but we will show a result which concerns rigidity for approximate solutions of the differential inclusion associated to stationary points, and we will prove that, for $g$ satisfying \eqref{HP}, we have
\[
(\KK)^{qc} = \KK.
\]
Namely, we will prove the following:
\begin{theoremint}\label{intro:detstro}
Let $\Omega \subset \R^2$ be open, bounded and convex. Let $u_n \in \Lip(\Omega,\R^6)$ be an equi-Lipschitz sequence with 
\[
\dist(Du_n(x),\KK) \to 0 \text{ in } L^1_{\loc}(\Omega).
\]
Suppose moreover that $u_n$ converges weakly-$*$ in $W^{1,\infty}$ to $u \in \Lip(\Omega,\R^6)$. Then, $\Det_{12}(Du_n)$ converges strongly in $L_{\loc}^1$ to $\Det_{12}(Du)$ and $Du \in K$.
\end{theoremint}
 
 \subsection{Notation}\label{Not}
 
 Throughout the paper, $(a,b)$ denotes the standard scalar product between vectors $a,b \in \R^n$ and $\langle A,B\rangle$ the Hilbert-Schmidt scalar product between matrices of $\R^{n\times m}$. The Euclidean norm of vectors $a \in \R^n$ and matrices $A \in \R^{n\times m}$ is denoted by $|a|$ and $|A|$. If $X \in \R^{n\times n}$, we denote with $\cof^T(X)$ the transpose of the matrix defined as
\[
\cof(X)_{ij} = (-1)^{i + j}\det(M_{ji}(X)),
\]
where $M_{ji}(X)$ denotes the $(n-1)\times(n-1)$ submatrix of $X$ obtained by eliminating from $X$ the $j$-th row and the $i$-th column. In particular, $\cof(X)$ satisfies
\[
X\cof(X) = \cof(X)X = \det(X)\id.
\] 
For a (Lebesgue) measurable set $E$, $|E|$ denotes its Lebesgue measure. For any set $D \subset \R^n$, $\overline{D}$ denotes its closure and $\partial D$ its topological boundary.

\subsection*{Aknowledgements} The author wishes to thank Prof. L\'aszl\'o Sz\'ekelyhidi for introducing him to this problem and Prof. Bernd Kirchheim for valuable discussions.

\section{Local invertibility results for Sobolev maps}\label{subs:inv}

We recall here \cite[Theorem 3.1]{FON}, that will prove to be one of the main ingredients of the proofs of Section \ref{s:exact}. We restate all the results of this section for Lipschitz maps, but both Theorems \ref{INV} and \ref{thm:quas} hold for $W^{1,n}$ mappings.

\begin{theorem}\label{INV}
Let $\Omega \subset \R^n$ be an open set, and let $\varphi: \Omega\to \R^n$ be a Lipschitz map with $\det(D\varphi) > 0$ a.e.. Then, $\varphi$ admits an a.e. local inverse, i.e. for a.e. $x_0 \in \Omega$, there exist $r,R > 0$ depending on $x_0$, an open neighborhood $D$ of $x_0$ and $w\in W^{1,1}(B_r(y_0),D)$, where $y_0 = \varphi(x_0)$, such that
\begin{enumerate}
\item $D = \varphi^{-1}(B_r(y_0))\cap B_R(x_0)$, $\overline{D} \subset \Omega$, $\varphi(\partial D) \subset \partial B_r(y_0)$ and $\varphi(D) = B_r(y_0)$; \label{carD}
\item $w\circ \varphi(x) = x$ for a.e. $x \in D$; \label{f}
\item $\varphi\circ w(y) = y$ for a.e. $y \in B_r(y_0)$;\label{s}
\item $Dw(y) = (D\varphi)^{-1}(w(y))$ for a.e. $y \in B_r(y_0)$;\label{t}
\item for every Lipschitz function $f: \overline{D} \to \R$, $f\circ w \in W^{1,1}(B_r(y_0))$ and the chain rule holds, i.e.\label{fo}
\[
D(f\circ w)(y) = Df(w(y))Dw(y), \text{ a.e. in }B_r(y_0).
\]
\end{enumerate}
\end{theorem}
\begin{proof}
The first property is not explicitely written in the statement of \cite[Theorem 3.1]{FON} but the definition of $D$ as in \eqref{carD} is given at the beginning of the proof of that Theorem. The existence and regularity of such $w$ and $\eqref{f},\eqref{s},\eqref{t}$ are the content of \cite[Theorem 3.1]{FON}. The validity of the chain rule \eqref{fo} can be deduced from \cite{ADM,LM}, or adapting the proof of \cite[Thm. 3.1, Claim 5]{FON}.
\end{proof}

Finally we recall here the definition and some property of quasiregular mappings. We say that $\varphi\in W_{\loc}^{1,n}(\Omega,\R^n)$ is quasiregular if there exists $K \ge 1$ such that
\[
|D\varphi|^n(x) \le K\det(D\varphi(x)).
\]
The result we are interested in is the following:

\begin{theorem}\label{thm:quas}
Let $\varphi$ be a nonconstant quasiregular Lipschitz mapping. Then, there exists a closed set $B_\varphi$ with $|B_\varphi| = 0$ and topological dimension at most $n - 2$ such that, on $\Omega\setminus B_\varphi$, $\varphi$ is a local homeomorphism. Moreover, if $x_0 \in \Omega\setminus B_\varphi$ and $B_r(x_0)$ is a ball where $\varphi$ is a homeomorphism, $\overline{B_r(x_0)} \subset \Omega$, then $\varphi^{-1}:\varphi(B_r(x_0)) \to B_r(x_0)$ is a $W_{\loc}^{1,n}(\varphi(B_r(x_0)))$ quasiregular mapping, with gradient
\[
D(\varphi^{-1}(y)) = (D\varphi)^{-1}(\varphi^{-1}(y)), \text{ a.e. in }\varphi(B_r(x_0)),
\]
and the chain rule holds, i.e. for every $f \in \Lip(\overline{B_r(x_0)})$, $f\circ \varphi^{-1} \in W^{1,n}_{\loc}(\varphi(B_r(x_0)))$ with
\[
D(f\circ\varphi^{-1})(y) = Df(\varphi^{-1}(y))(D\varphi)^{-1}(\varphi^{-1}(y)), \text{ a.e. in }\varphi(B_r(x_0)).
\]
\end{theorem}
\begin{proof}
The existence of such a \emph{branch} set $B_\varphi$ and its properties can be found in \cite[Ch. II.6 and II.10]{RES}. The fact that the inverse of a quasiregular homeomorphism, which is usually named a quasiconformal map, is itself quasiregular can be found in \cite[Theorem 9.1]{BI}. For the proof of the chain rule, we refer the reader to \cite[Lemma 9.6]{BI}.
\end{proof}

\section{Exact solutions}\label{s:exact}

In this Section we want to prove one of our main results, Theorem \ref{thm:poly}. We will deduce it from the following:

\begin{theorem}\label{lem:const}
Let $\Omega\subset \R^m$ be open and connected and let $n \ge m\ge 2$. Let $u\in \Lip(\Omega,\R^n)$ and $\beta \in L^\infty(\Omega)$ satisfy in the weak sense
\begin{equation}\label{eq:curl}
\curl(\beta Du) = 0.
\end{equation}
Suppose that there exists an $m \times m$ minor $M$ such that
\begin{equation}\label{eq:pos}
\beta\Det(M(Du)) > 0, \text{ a.e. in }\Omega.
\end{equation}
Then, $\beta$ is \emph{essentially} locally constant, in the sense that there exists an open set $\Omega_0$ with $|\Omega\setminus\Omega_0| = 0$ such that, on $\Omega_0$, $\beta$ is locally constant. Moreover, if there exists $\delta > 0$ such that
\begin{equation}\label{eq:delta}
\beta\Det(M(Du)) \ge \delta, \text{ a.e. in } \Omega,
\end{equation}
 then, $\beta$ is constant. 
\end{theorem}

\begin{remark}
An additional assumption as \eqref{eq:pos} is needed: as stated in \cite[Question 10]{KMS}, given a ball $\Omega \subset \R^2$, it is possible to construct a measurable function $\beta: \Omega \to \{-1,1\}$ and $u \in \Lip(\Omega,\R^2)$ with $\Det(Du): \Omega \to \{-1,1\}$ such that in the sense of distributions
\[
\curl(\beta Du) = 0,
\]
but $\beta$ is not constant on any open subset of $\Omega$.
\end{remark}

In Subsection \ref{subsec:transport}, we will give an alternative (and informal) proof of the previous theorem in a simplified case. This will highlight some geometric features of the assertion, that are hidden in the following proof.

\begin{proof}[Proof of Theorem \ref{lem:const}]
Without loss of generality, we can suppose that $m = n$ and therefore \eqref{eq:pos} reads as
\begin{equation}\label{eq:poss}
\beta\Det(Du) > 0.
\end{equation}
Moreover, up to restricting to balls we can assume that $\Omega$ is convex, thus \eqref{eq:curl} is equivalent to
\[
\beta Du = Dv,
\]
for some $v \in \Lip(\Omega,\R^n)$. Let
\[
u = \left(\begin{array}{c}
u_1\\
u_2\\
\dots\\
u_n
\end{array}
\right)
\text{ and }
v = \left(\begin{array}{c}
v_1\\
v_2\\
\dots\\
v_n
\end{array}
\right).
\]
We introduce the two maps 
\[
\varphi(x) \doteq
\left(
\begin{array}{c}
v_1\\
u_2\\
u_3\\
\dots\\
u_n
\end{array}
\right),\quad
\psi(x) \doteq
\left(
\begin{array}{c}
u_1\\
v_2\\
v_3\\
\dots\\
v_n
\end{array}
\right).
\]
We have that
\begin{equation}\label{Dfi}
D\varphi =
D\left(
\begin{array}{c}
v_1\\
u_2\\
u_3\\
\dots\\
u_n
\end{array}
\right)=
\left(
\begin{array}{c}
\beta Du_1\\
Du_2\\
Du_3\\
\dots\\
Du_n
\end{array}
\right)
\end{equation}
This yields $\det(D\varphi) = \det(Du)\beta(x) >0$, from hypothesis \eqref{eq:pos}. Let $x_0 \in \Omega$ be a point at which Theorem \ref{INV} applies. This means that we find $r > 0$, a compactly contained open set $D \subset \Omega$ and $w \in W^{1,1}(B_r(y_0))$, a.e. local inverse of $\varphi$ in the sense of Theorem \ref{INV}. Let $B\doteq B_r(\varphi(x_0))$ be the domain of $w$. Theorem \ref{INV}\eqref{t} shows that
\[
Dw(y) = (D\varphi)^{-1}(w(y)), \quad \text{ for a.e. }y \in B.
\]
Consider the map $g(y)\doteq \psi\circ \varphi^{-1}(y)$. By Theorem \ref{INV}\eqref{fo}, this is a $W^{1,1}(B,\R^n)$ map, whose differential is given a.e. in $B$ by the matrix
\begin{equation}\label{Dg}
Dg(y) = D\psi(w(y))Dw(y) = D\psi(w(y))(D\varphi)^{-1}(w(y))
\end{equation}
We compute explicitely the matrix-field
\[
z\mapsto D\psi(z)(D\varphi)^{-1}(z).
\]
We have
\[
D\psi = 
D\left(
\begin{array}{c}
u_1\\
v_2\\
v_3\\
\dots\\
v_n
\end{array}
\right)=
\left(
\begin{array}{c}
Du_1\\
\beta Du_2\\
\beta Du_3\\
\dots\\
\beta Du_n
\end{array}
\right).
\]
Consider the set $D' \subset D$ of points $z$ such that $z$ is a Lebesgue point for $\beta, Du, D\psi, D\varphi$ and for $(D\varphi)^{-1}$, $Du(z)$ and $D\varphi(z)$ are invertible and $\beta(z) \neq 0$. Since all the functions under consideration are $L^1_{\loc}$ and by \eqref{eq:pos}, $|D\setminus D'| = 0$. At any such (fixed) point $z$, denote the columns of $(Du(z))^{-1}$ by $Y_1,\dots, Y_n \in \R^n$, i.e.
\[
(Du(z))^{-1} = (Y_1|\dots |Y_n).
\]
Notice that $Y_i$ depends on $z$, but we drop the dependence since we anyway want to make a computation for fixed $z$. We claim that
\begin{equation}\label{claiminv}
(D\varphi)^{-1}(z) = \left(\frac{1}{\beta(z)}Y_1|\dots | Y_n\right) \doteq N(z).
\end{equation}
Indeed, notice that $(Du_i(z),Y_j) = \delta_{ij}$ for all $1\le i,j\le n$, where $\delta_{ij}$ denotes Kronecker's delta. Thus, combining \eqref{Dfi} with the expression of $N(z)$
\[
(D\varphi(z)N(z))_{ij} =
\begin{cases}
\left(\beta(z)Du_1,\frac{1}{\beta(z)}Y_1\right) = 1, &\text{ if } i = j = 1\\
\left(\beta(z)Du_i,Y_1\right) = 0, &\text{ if } i \neq 1, j= 1\\
\left(\beta(z)Du_1,Y_j\right) = 0, &\text{ if } i = 1, j \neq 1\\
\left(Du_i(z),Y_j\right) = \delta_{ij}, &\text{ otherwise}.
\end{cases}
\]
Hence \eqref{claiminv} is checked. Now we can easily compute
\[
(D\psi(z)(D\varphi)^{-1}(z))_{ij} =(D\psi(z)N(z))_{ij} =
\begin{cases}
\left(Du_1,\frac{1}{\beta(z)}Y_1\right) = \frac{1}{\beta(z)}, &\text{ if } i = j = 1\\
\left(\beta(z)Du_i,\frac{1}{\beta(z)}Y_1\right) = 0, &\text{ if } i \neq 1, j= 1\\
\left(Du_1,Y_j\right) = 0, &\text{ if } i = 1, j \neq 1\\
\left(\beta(z)Du_i(z),Y_j\right) = \beta(z)\delta_{ij}, &\text{ otherwise}.
\end{cases}
\]
In other words, for a.e. $z \in D$ we have
\[
D\psi(z)(D\varphi(z))^{-1}=
\left(
\begin{array}{cc}
\begin{matrix}
    1/\beta(z) &0 &0 & 0 & \dots  & 0 \\
    0 & \beta(z) &0 & 0 & \dots&0 \\
    0 & 0 & \beta(z) & 0 &\dots & 0 \\
    0 & 0 & 0& \beta(z) & \dots & 0\\
    \dots & \dots &  \dots& \dots & \dots & \dots\\
    0 & 0 &0 & 0& \dots  &  \beta(z)
\end{matrix}
\end{array}
\right).
\]
Let $A(y) \doteq D\psi(w(y))(D\varphi(w(y)))^{-1}$. Recalling \eqref{Dg}, we have that $\curl(A) = 0$ in the sense of distributions. For a diagonal matrix-field $M = \diag(m_1(y),\dots, m_n(y))$, having $\curl(M) = 0$ is equivalent to say that $m_i(y)$ only depends on $y_i$. Thus, due to the special form of $A$, we obtain that $y \in B \mapsto \beta(w(y))$ is constant. By Theorem \ref{INV}\eqref{carD}-\eqref{f} we know that $w$ is essentially surjective, in the sense that $|w(B)\setminus D| = 0$. Therefore, $\beta$ is constant on $D$. If we let 
\begin{equation}\label{eq:E}
E =\{x \in \Omega: x \text{ admits an a.e. local inverse in a neighborhood } D(x)\},
\end{equation}
 we can consider
\[
\Omega_0 \doteq \bigcup_{x \in E}D(x).
\]
By Theorem \ref{INV}, $|\Omega\setminus E| = 0$, hence $|\Omega\setminus \Omega_0| = 0$ and, by the proof above, on $\Omega_0$ the function $\beta$ is locally constant.
\\
\\
To conclude the proof of the present Theorem, we need to show \eqref{eq:delta}. To do so, it is sufficient to notice that \eqref{eq:delta}, the Lipschitz property of $u$ and $\beta \in L^\infty$ imply that $\varphi$ is a quasiregular mapping. The proof above can be applied again in the neighborhood of every $x_0$ where $\varphi$ is an homeomorphism, since by Theorem \ref{thm:quas} in those neighborhoods we have a.e. local invertibility in the sense of Theorem \ref{INV}. In this case, though, we have the additional information that the set $E$ introduced in $\eqref{eq:E}$ has complementary with topological dimension $\le n - 2$. Thus $E^c$ does not disconnect $\Omega$, and we conclude that $\beta$ is constant on $\Omega$.
\end{proof}

Theorem \ref{lem:const} yields as a corollary Theorem \ref{thm:poly}, that we recall here.
\begin{corollary}\label{thm:polysec}
Let $\Omega \subset \R^2$ be open and connected. Let $u \in \Lip(\Omega,\R^2)$ and $g$ satisfy \eqref{HP}. Suppose that \eqref{EL} holds in the weak sense and that
\[
\det(Du) \neq 0, \text{ a.e. in }\Omega.
\]
Then, $\det(Du)$ is \emph{essentially} locally constant, in the sense that there exists an open set $\Omega_0$ with $|\Omega\setminus\Omega_0| = 0$ such that, on $\Omega_0$, $x\mapsto \det(Du)$ is locally constant. If, moreover, there exists $\delta > 0$ such that $\det^2(Du)\ge \delta$, then $\det(Du)$ is constant in $\Omega$.
\end{corollary}

\begin{proof}
Define $\beta(x)\doteq g'(\det(Du))(x)$. We want to apply Theorem \ref{lem:const} with $n = m = 2$. In this case, $M(A) = A$, $\forall A \in \R^{2\times 2}$. We only need to prove that if $\det(Du) \neq 0$, then \eqref{eq:pos} holds, and if $\det(Du)^2\ge \delta$, then $\eqref{eq:delta}$ holds. Notice that because of the hypotheses on $g$ given by \eqref{HP}, then $g$ is uniformly convex on compact set. If we assume that $|\det(Du)|(x) \le L$, for a.e. $x \in \Omega$, then we find $\alpha = \alpha(L) > 0$ such that $g''(t) \ge \alpha, \forall |t|\le L$. Therefore, the convexity of $g$ and the assumption $g'(0) = 0$ yield
\[
g'(t)t = (g'(t)-g'(0))(t - 0) \ge \alpha t^2,\quad \forall |t| \le L,
\]
and hence
\[
\beta(x)\det(Du)(x) = g'(\det(Du)(x))\det(Du(x)) \ge \alpha \det(Du(x))^2,\quad \text{a.e. in }\Omega.
\]
Therefore $\det(Du) \neq 0$ a.e. and $\det(Du)^2\ge \delta$ imply \eqref{eq:pos} and \eqref{eq:delta} respectively (where the $\delta$ of Theorem \ref{lem:const} has to be substituted with $\alpha\delta$ > 0). Theorem \ref{lem:const} finishes the proof.
\end{proof}

\subsection{Another proof of Theorem \ref{lem:const}}\label{subsec:transport}

In this subsection we wish to give a different proof of Theorem \ref{lem:const} in a simplified setting. We believe that this highlights some geometrical features of the problem which are hidden in the proof of Theorem \ref{lem:const}. Let $\Omega \subset \R^2$ be an open and convex bounded domain. Assume $u \in  \Lip(\Omega,\R^2)$ is a solution to
\[
\dv(\Det(Du)\cof^T(Du)) = 0,
\]
i.e. to
\begin{equation}\label{eq:usual}
\curl(\Det(Du)Du) = 0.
\end{equation}
Let us further suppose that $\Det(Du) \in \{1,-1\}$ a.e. in $\Omega$, which is analogous to condition \eqref{eq:pos}. Our aim is to show that $\Det(Du)$ is constant on $\Omega$. Since $\Omega$ is convex, \eqref{eq:usual} implies that there exists $v: \Omega\to\R^2$ such that, a.e. in $\Omega$,
\begin{equation}\label{eq:def}
\det(Du)Du = Dv.
\end{equation}
As in Theorem \ref{lem:const}, the key is to consider the maps
\[
\varphi(x) \doteq
\left(
\begin{array}{c}
u_1\\
v_2
\end{array}
\right),\quad
\psi(x) \doteq
\left(
\begin{array}{c}
v_1\\
u_2\\
\end{array}
\right).
\]
For a.e. $x \in \Omega$, $\det(D\varphi)(x) = \det(D\psi)(x) = 1$. This and the assumption that $\varphi$ and $\psi$ is Lipschitz imply that $\varphi$ and $\psi$ are quasiregular mappings, therefore we can use Theorem \ref{thm:quas} to deduce that in the neighborhood of a.e. point $\varphi$ and $\psi$ are homeomorphisms. Let $x_0$ be such a point and $B$ denote its neighborhood. We need now to consider the level sets of $u_1,u_2,v_1$ and $v_2$. Assume for the moment that a connected component of $\{x \in B: u_1(x) = u\}$ is parametrized by an injective Lipschitz curve $\gamma:[0,1]\to \R^2$ with $\gamma'(t) \neq 0$ for a.e. $t$, on which $u_i\circ \gamma$ and $v_i\circ \gamma$ are differentiable a.e. and the chain rule holds, i.e. for a.e. $t \in [0,1]$,
\begin{equation}\label{chaincurve}
\frac{d}{dt}u_i(\gamma(t)) = (Du_i(\gamma(t)),\gamma'(t)), \quad \frac{d}{dt}v_i(\gamma(t)) = (Dv_i(\gamma(t)),\gamma'(t)).
\end{equation}
Since $\gamma$ parametrizes a level set of $u_1$, \eqref{chaincurve} implies
\begin{equation}\label{derivaort}
\gamma'(t) = \lambda(t)JDu_1\circ \gamma(t) \text{ for a.e. }t \in [0,1],
\end{equation}
where $\lambda\in L^\infty([0,1])$, $\lambda(t) \neq 0$ for a.e. $t$, and $J$ is the rotation of $90$ degrees chosen in such a way that for every matrix $A \in \R^{2\times 2}$ with rows $A_1,A_2$, $(A_2,JA_1) = \det(A)$. Moreover, \eqref{eq:def} is telling us that $Dv_1$ is parallel to $Du_1$. Hence there exists $v \in \R$ such that $v\circ \gamma(t) \equiv v$. Anyway, as noticed above, $\varphi$ and $\psi$ are injective in $B$. Therefore, $u_1\circ \gamma \equiv u$ and $\varphi$ injective implies $v_2\circ \gamma$ injective, and $v_1\circ \gamma \equiv v$ and $\psi$ injective implies $u_2\circ\gamma$ injective. In particular, $u_2\circ\gamma$ and $v_2\circ\gamma$ are monotone and thus their derivatives are non-negative or non-positive. Hence, using that $\Det^2(Du) = 1$ a.e.,
\[
\frac{d}{dt}v_2\circ\gamma(t) = (Dv_2(t),\gamma'(t)) \overset{\eqref{eq:def}-\eqref{derivaort}}{=} \det(Du)\lambda(t)(Du_2,JDu_1)(\gamma(t)) = \lambda(t)\Det^2(Du) = \lambda(t)
\]
and
\[
\frac{d}{dt}u_2\circ\gamma(t) = (Du_2(t),\gamma'(t)) = \lambda(t)(Du_2,JDu_1)(\gamma(t)) = \lambda(t)\det(Du)\circ \gamma(t).
\]
Since these derivatives must have a sign and $\det(Du) \in \{1,-1\}$, we find that $\det(Du)\circ \gamma$ must be constant. We could make the same reasoning on the level curves of $u_2$, to obtain similarly that $\det(Du)$ is constant on Lipschitz curves parametrizing level sets of $u_2$. Heuristically speaking, since $\det(Du) \neq 0$ a.e., the level sets $\{u_1 = u\}$ and $\{u_2 = u'\}$ must intersect transversely, and one can expect that this and the argument above shows $\Det(Du) \equiv c$ on $B$. To this end, we consider as a new parametrization $\varphi^{-1}(y)$. Indeed, if we consider a small open cube $Q = (a,b)^2$ centered at $\varphi(x_0)$ in such a way that $Q \subset \varphi(B)$, then the open neighborhood of $x_0$ given by $\varphi^{-1}(Q)$ is naturally foliated by Lipschitz curves where $u_1$ is constant, corresponding to $s\mapsto \varphi^{-1}(t,s)$, and Lipschitz curves where $v_2$ (and hence $u_2$) is constant, corresponding to $t\mapsto \varphi^{-1}(t,s)$. The reasoning above implies that for a.e. $t \in \R$, the determinant remains constant on the independent foliations of $\varphi^{-1}(Q)$ in Lipschitz curves of the form $s\mapsto \varphi^{-1}(t,s)$ and $t\mapsto \varphi^{-1}(t,s)$, but it is now immediate to see that this implies the constancy of the determinant on the whole open set $\varphi^{-1}(Q)$.
\\
\\
This proof does not work if $\det^2(Du)$ is not constant, but anyway provides us with the right elements to use in the more general case. In particular, it shows the geometrical importance of the new parametrization $\varphi^{-1}$, which was also used in an essential way in the proof of Theorem \ref{lem:const}.

\subsection{Higher dimensions}

In this subsection we show Theorem \ref{thm:poly} in every dimension, requiring anyway stronger assumptions on the solutions, compare \eqref{eq:poshigh} with \eqref{notzeroint}.

\begin{theorem}\label{thm:polygensec}
Let $\Omega \subset \R^n$ be open and connected. Let $u \in \Lip(\Omega,\R^n)$ and $g$ satisfy \eqref{HP}. Suppose that \eqref{EL} holds in the weak sense and that
\begin{equation}\label{eq:poshigh}
\det(Du) > 0, \text{ a.e. in }\Omega.
\end{equation}
Then, $\det(Du)$ is \emph{essentially} locally constant, in the sense that there exists an open set $\Omega_0$ with $|\Omega\setminus\Omega_0| = 0$ such that, on $\Omega_0$, $x\mapsto \det(Du)$ is locally constant. If, moreover, there exists $\delta > 0$ such that $\det(Du)\ge \delta$, then $\det(Du)$ is constant in $\Omega$.
\end{theorem}

\begin{proof}
We again invoke Theorem \ref{INV}, and we fix a point $x_0$ such that in an open neighborhood $D= D(x_0)$ the map $u$ is almost invertible with almost inverse $w: D \to B_r(u(x_0))$. Let $B\doteq B_r(u(x_0))$. The Euler-Lagrange equations \eqref{EL} read as
\[
0 = \int_Dg'(\det(Du))\langle \cof^T(Du),D\eta \rangle dx, \quad \forall \eta \in C_c^\infty(D,\R^n).
\]
A standard approximation argument shows that this holds for $\eta \in W^{1,1}_0(D,\R^n)$. As we will need to use the chain rule of Theorem \ref{INV}\eqref{fo}, we fix any $\eta \in W^{1,\infty}_0(D;\R^n)$. We use \eqref{EL} and the area formula to find
\begin{align*}
0 &= \int_Dg'(\det(Du))\langle \cof^T(Du),D\eta \rangle dx= \int_Df(x)\det(Du)(x)\langle \cof^T(Du),D\eta \rangle dx \\
&= \int_{B}f(w(y))\langle \cof^T(Du)(w(y)), D\eta(w(y))\rangle dy,
\end{align*}
where $f(x) \doteq \frac{g'(\det(Du))(x)}{\det(Du)(x)}$.
Using Theorem \ref{INV}\eqref{t}, we rewrite
\[
 \cof^T(Du)(w(y)) =  \det(Du)(w(y))(Du)^{-T}(w(y)) = \det(Du)(w(y))Dw(y)^T,
\]
where $X^{-T} = (X^{-1})^T$ for every invertible $X \in \R^{n\times n}$. Therefore, continuing our chain of equalities,
\begin{align*}
0 &= \int_{B}f(w(y))\langle \cof^T(Du)(w(y)), D\eta(w(y))\rangle dy = \int_{B}f(w(y))\det(Du)(w(y))\langle Dw(y)^T, D\eta(w(y))\rangle dy \\
&= \int_{B}g'(\det(Du)(w(y)))\tr(D\eta(w(y))Dw(y)) dy = \int_{B}g'(\det(Du)(w(y)))\tr(D(\eta\circ w)(y)) dy \\
&= \int_{B}g'(\det(Du)(w(y)))\diver(\eta\circ w)(y) dy.
\end{align*}
Notice that in the second line we used the chain rule of Theorem \ref{INV}\eqref{fo}. If we manage to show that for every $\Phi \in W^{1,\infty}_0(B,\R^n)$, there exists $\eta \in W^{1,\infty}_0(D,\R^n)$ such that
\begin{equation}\label{eq:surj}
\Phi = \eta\circ w,
\end{equation}
then the previous equation would tell us that $g'(\det(Du)) \equiv C \in \R$, and, by the injectivity of $t\mapsto g'(t)$, we would conclude the proof. We only need to show \eqref{eq:surj}. To do so, we fix $\Phi \in W^{1,\infty}_0(B,\R^n)$ and we observe that \eqref{eq:surj} is uniquely solved by $\eta \doteq \Phi\circ u$. By our assumptions, both $\Phi$ and $u$ are Lipschitz, and hence so is $\eta$. Therefore, $\eta \in W^{1,\infty}(D,\R^n)$. We only need to check that $\eta(x) = 0$, $\forall x \in \partial D$. This would imply $\eta \in W^{1,\infty}_0(D,\R^n)$ since $\eta$ is continuous. To show this, we use Theorem \ref{INV}\eqref{f}, to find that $u(\partial D ) \subset \partial u(D) = \partial B$. Thus if $x \in \partial D$, then $u(x) \in \partial B$, and hence $\Phi(u(x)) = 0$ since $\Phi \in W^{1,\infty}_0(B,\R^n)$. It follows that $\eta \in W^{1,\infty}_0(D,\R^n)$ and the proof is concluded.
\\
\\
Now the way to conclude is completely analogous to the conclusion of the proof of Theorem \ref{lem:const}. Finally, if we assume the stronger condition $\det(Du) \ge \delta$ a.e., then we can again apply the quasiregular mappings theory and in particular Theorem \ref{thm:quas} to find that $\det(Du)$ is globally constant in $\Omega$.
\end{proof}

\subsection{Stationary points and solutions with higher regularity}\label{subs:higher}

In this subsection we collect some results similar to the one stated in Corollary \ref{thm:polysec} with additional assumptions on the regularity of the solution $u$. In particular, we start by proving Theorem \ref{cor:poly}, that we first recall.

\begin{corollary}\label{cor:polysec}
Let $\Omega \subset \R^2$ be open and connected and let $u\in \Lip(\Omega,\R^2)$ be a stationary point of the energy \eqref{EN}. Then, $x \mapsto \det(Du)(x)$ is constant.
\end{corollary}
\begin{proof}
We recall that stationarity for $u$ means that $u$ satisfies system \eqref{vargr}. The inner variation equations for $f(X) = g(\det(X))$ are
\begin{align*}
0 = \diver(Du^TDf(Du) - f(Du)\id) &= \diver(g'(\det(Du))Du^T\cof(Du)^T - g(\det(Du)\id)) \\
&= \diver((g'(\det(Du)\det(Du) - g(\det(Du)))\id),
\end{align*}
that has to be intended in the weak sense. In the previous chain of equalities we have used that $\cof(X)^T = \cof(X^T)$ and $\det(X^T) = \det(X)$, $\forall X \in \R^{2\times 2}$. For any $f \in L^1(\Omega,\R)$,
\[
\diver(f\id) = Df
\]
in the sense of distributions. Therefore the inner variations equations tell us that the function $$x\mapsto g'(\det(Du(x)))\det(Du)(x) - g(\det(Du)(x))$$ is constant, so there exists $C \in \R$ such that
\[
g'(\det(Du(x)))\det(Du)(x) - g(\det(Du)(x)) \equiv C.
\] Hypothesis \eqref{HP} moreover implies that for every $C\neq 0$, the equation $g'(t)t - g(t) = C$ has at most two solutions $t_1, t_2 \in \R\setminus\{0\}$, and if $C = 0$, then $t = 0$. To see this, we only need to take a derivative of $h(t)\doteq g'(t)t - g(t)$:
\[
h'(t) = g''(t)t +g'(t) - g'(t) = g''(t)t,
\]
and by the strict convexity of $g$ we have $h'(t) > 0$ for $t > 0$ and $h'(t) < 0$ for $t < 0$. To conclude that for $C = 0$ the only solution to $h(t) = 0$ is $t = 0$ we just need to observe that $h(0) = 0$, again by \eqref{HP}. Therefore, if $C = 0$, then $\det(Du(x)) = 0$ for a.e. $x \in \Omega$, and the proof is finished. Otherwise, $\det(Du(x)) = t_1\neq 0$ or $\det(Du(x)) = t_2\neq 0$ for a.e. $x \in \Omega$. We can therefore apply Theorem \ref{thm:polysec} to conclude that  $x\mapsto \det(Du(x))$ is constant in $\Omega$.
\end{proof}

We now consider $C^1$ and $W^{2,1}\cap W^{1,\infty}$ solutions to \eqref{EL}. These proofs are simpler and, at least in the $C^1$ case, probably known to experts, but we include the proofs for completeness.
\begin{corollary}\label{cor:C1}
Let $\Omega$ be open and connected and let $u\in C^1(\Omega,\R^n)$ be a critical point of the energy \eqref{EN}. Then, $x \mapsto \det(Du)(x)$ is constant.
\end{corollary}
\begin{proof}
There are two cases. First, $\det(Du(x)) = 0$ for every $x \in \Omega$. In this case there is nothing to prove. Otherwise, there exists a point $x_0$ such that $\det(Du(x_0)) = C \neq 0$. Now it suffices to notice that the nonempty set
\[
E = \{x \in \Omega: \det(Du(x)) = C\}
\]
is both open and closed relatively to $\Omega$. Indeed, the fact that it is closed is immediate from the continuity of $x \mapsto \det(Du(x))$. Moreover, the set is open, indeed let $z \in E$. Then, $\det(Du(z)) \neq 0$, and hence the classical inverse function Theorem applies. It follows that, in a neighborhood $B$ of $z$, $x\mapsto \det(Du(x))$ is constant as can be seen following for instance the same lines of the proof of Theorem \ref{thm:polygensec}. As $\det(Du(z)) = C$, it follows $\det(Du(x)) = C$ for all $x \in B$, hence $B \subset E$. We deduce $E = \Omega$ by the connectedness of $\Omega$.
\end{proof}

\begin{corollary}\label{cor:W2}
Let $\Omega$ be open and connected. Let $u\in \Lip\cap W^{2,1}(\Omega,\R^n)$ be a critical point of the energy \eqref{EN}. Then, $x \mapsto \det(Du)(x)$ is constant.
\end{corollary}
\begin{proof}
We invoke \cite[Theorem 1]{DETREM}, which tells us that if $\Omega\subset \R^n$ is an open set, $\sigma \in L^q(\Omega,\R^n)$, $\dv(\sigma) \in L^1(\Omega)$, $v \in W^{1,p}(\Omega)$ for $1\le p < n$ with
\begin{equation}\label{este}
\frac{1}{p} - \frac{1}{n} + \frac{1}{q} \le 1
\end{equation}
and the distribution $\diver(v\sigma)$ is represented by $d \in L^1$, then
\begin{equation}\label{eq:diver}
d(x) = (Dv(x),\sigma(x)) + v(x)\diver(\sigma)(x), \quad \text{for a.e. $x \in \Omega$.}
\end{equation}
In our case, $v(x) = g'(\det(Du(x))) \in W^{1,1}\cap L^\infty(\Omega)$, $\sigma = (\cof(Du)^T)_i$, where $i \in \{1,\dots, n\}$ and $ (\cof(Du)^T)_i$ is the $i$-th row of $\cof^T(Du)$, and $d(x) \equiv 0$. Notice that under our hypotheses \eqref{este} holds. Since the rows of $\cof^T(Du)$ are divergence free, \eqref{eq:diver} reads as
\[
\cof^T(Du)D(g'(\det(Du))) = 0
\]
in our case. Therefore, $D(g'(\det(Du))(x) = 0$ a.e. in the set
\[
E\doteq \{x: \det(Du) \neq 0\} \subset \Omega.
\]
If $E$ is of full measure in $\Omega$, then the proof is finished by the injectivity of $t \mapsto g'(t)$, that is a consequence of \eqref{HP}. By the so-called \emph{strong locality property} of the Sobolev derivatives, see \cite[Theorem 4.4(iv)]{EVG}, if $|E^c|\neq 0$, we anyway find that a.e. on $E^c$ $D(g'(\det(Du))) = 0$, since on $E^c$ the function $x\mapsto \det(Du(x))$ is constant. This implies that $D(g'(\det(Du))) = 0$ a.e. in $\Omega$ and finishes the proof.
\end{proof}

\section{Approximate Solutions}\label{s:approx}

In this section we study rigidity of approximate solutions for solutions of the system \eqref{vargr}.

\subsection{Young measures and differential inclusions}\label{subs:young}

First we recall the Fundamental Theorem on Young measures for equibounded sequences in $L^p$. We refer the reader to \cite[Section 3]{DMU} for a complete exposition on the subject. The results we report here are taken from that reference. In the following, $\mathcal{M}(\R^m)$ denotes the space of finite and positive measures on $\R^m$, $\mathcal{P}(\R^m)$ the space of probability measures on $\R^m$. Moreover, $\mathcal{L}^d$ denotes the Lebesgue measure on $\R^d$. Finally, if $\mu \in \mathcal{M}(\R^{N})$ and $f \in L^1(\R^N;\mu)$,
\[
\langle\mu,f\rangle \doteq \int_{\R^{N}}f(y)d\mu(y).
\]

\begin{theorem}[Fundamental Theorem on Young measures]
Let $E\subset \R^d$ be a Lebesgue measurable set with finite measure and let $p > 1$. Consider a sequence $z_j:E \to \R^N$ of measurable functions weakly converging in $L^p$ to some function $z$. Then, there exists a subsequence $z_{j_k}$ that \emph{generates the Young measure} $\nu$, where $\nu$ is a weak-* measurable map $\nu:E \to \mathcal{M}(\R^N)$ such that for $\mathcal{L}^d$-a.e. $x \in E$, $\nu_x \in \mathcal{P}(\R^N)$. The family $\nu = (\nu_x)_{x \in E}$ has the property that for every $f \in C(\R^N)$ such that
\[
|f(y)|\le C(1 + |y|^q), \text{ for } q < p,
\]
the following holds
\[
f(z_j) \rightharpoonup \bar{f}, \text{weakly in } L^{\frac{p}{q}}(E).
\]
In particular, the choice $f(y) = y,\; \forall y \in \R^N$ yields
\begin{equation}\label{exp}
z(x) = \langle\nu_x,f\rangle.
\end{equation}
\end{theorem}
One of the most useful results concerning Young measures is \cite[Corollary 3.2]{DMU}:
\begin{corollary}\label{strongc}
Suppose that a sequence $z_j \in L^p(E;\R^N)$ converges weakly to $z \in L^p(E;\R^N)$ and generates the Young measure $\nu = (\nu_x)_{x \in E}$. Then,
\[
z_j \to z \text{ in $L^q,\forall 1\le q < p$ if and only if } \nu_x = \delta_{z(x)} \text{ for $\mathcal{L}^d$-a.e }x.
\]
\end{corollary}
If the sequence $z_j = Du_j$, then we will call $\nu$ a gradient Young measure. Moreover, we will call a Young measure $\nu = (\nu_x)_{x \in \Omega}$ \emph{homogeneous} if $\nu_x$ does not depend on $x$. In particular, given a (gradient) Young measure $\nu = (\nu_x)_{x \in \Omega}$, for almost every $x \in \Omega$ there exists an homogeneous (gradient) Young measure $\mu = (\mu_y)_{y \in \Omega}$ in $\Omega$ such that $\mu_y = \nu_x$ for a.e. $y \in \Omega$, see \cite[Theorem 2.3]{KIP}.
\subsection{Approximate solutions}\label{APPSOL}
As recalled in the introduction, a differential inclusion is a relation of the kind
\[
Du(x) \in K,
\]
for a.e. $x \in \Omega$, where $K$ is a compact subset of $\R^{n\times m}$ and $u \in \Lip(\Omega,\R^n)$, $\Omega \subset \R^m$ open, bounded and convex. An important question is for which sets $K$ one has \emph{compactness for approximate solutions}, i.e. given a sequence $(u_n)_{n \in \N}$ with $\sup_n\|u_n\|_{W^{1,\infty}} \le C$ and
\begin{equation}\label{distzero}
\dist(Du_n,K) \to 0 \text{ in }L^1,
\end{equation}
is it true that $(Du_n)_{n \in \N}$ converges strongly in $L^1$ to some limit? \v Sver\'ak in \cite{SAP} gave conditions on $K\subset \R^{2\times 2}$ for this property to hold, i.e.
\[
c\det(X - Y) > 0, \text{ for some } c \in \R,\quad \forall X,Y \in K.
\]
We refer the reader to \cite{LASLOC,TR, KF} for related results and generalizations to higher dimensions. In order to study limits of equi-Lipschitz sequences of maps $u_n$ satisfying \eqref{distzero}, we consider the Young measure they generate, $\nu = (\nu_x)_{x \in \Omega}$. Condition \eqref{distzero} implies that
\[
\spt(\nu_x) \subset K, \text{ for a.e. }x \in \Omega.
\]
In \cite{KIP}, D. Kinderlehrer and P. Pedregal proved that gradient Young measures supported in a set $K$ are exactly those measures that satisfy Jensen's inequality with the class of quasiconvex function. We say that $f: \R^{n\times m} \to \R$ is quasiconvex if
\[
\fint_{\Omega}f(A + D\Phi(x))dx \ge f(A),\quad \forall A \in \R^{n\times m}, \Phi \in C^\infty_c(\Omega,\R^n).
\]
This yields duality between gradient Young measures and quasiconvex functions. From the collection of homogeneous Young measure supported in $K$, we can define
\[
K^{qc}\doteq\{X \in \R^{n\times m}: X = \langle \nu, \id \rangle, \nu \text{ homogeneous Young measures with } \spt(\nu)\subset K\}.
\]
By \cite[Theorem 4.10(iii)]{DMU}, one also finds the equivalent characterization:
 \begin{equation}\label{qc}
K^{qc}\doteq\{X \in \R^{n\times m}: f(X) \le 0, \forall \text{ quasiconvex }f \text{ s.t. }f|_{K}\le 0\}.
\end{equation}
Verifying that a function is quasiconvex is usually a hard problem, see for instance \cite{SQU}, therefore one introduces other classes of functions to get \emph{estimates} on the quasiconvex hull. In particular, one considers rank-one convex functions $f$, that are characterized by the property that
\[
t\mapsto f(A + tB) \text{ is convex }, \forall A,B \in \R^{n\times m}, \rank(B) = 1,
\]
and polyconvex functions, i.e. convex functions of the minors of $X$. Consequently, for a compact $K$, one defines $K^{rc}$ and $K^{pc}$ by separation using rank-one convex functions and polyconvex functions respectively, as in \eqref{qc}. One can prove that
\[
\text{$f$ polyconvex $\Rightarrow$ $f$ quasiconvex $\Rightarrow$ $f$ rank-one convex},
\]
therefore the following chain of inequality always holds:
\begin{equation}\label{incconv}
K\subseteq K^{rc} \subseteq K^{qc} \subseteq K^{pc}.
\end{equation}
By Corollary \ref{strongc}, it is clear that a set $K$ is compact for approximate solutions if and only if the set of homogeneous Young measures supported in $K$ consists only of Dirac deltas. In that case, it is immediate to see that $K^{qc} = K$. On the other hand, if $K^{rc}$ is sufficiently \emph{large}, one expects the existence of pathological solutions to the differential inclusion via convex integration, as in \cite{SMVS,LSP}. 
\subsection{On the differential inclusion associated to \eqref{vargr}}
Let us start by fixing the sets defining our differential inclusions. For $g$ satisfying $\eqref{HP}$ and $J = \left(\begin{array}{cc}
0& -1\\
1 & 0
\end{array}\right)$, we set:
\begin{align*}
K^g \doteq 
\left\{A \in \R^{4\times 2}:
A =
\left(
\begin{array}{c}
X\\
g'(\Det(X))X
\end{array}
\right)
\right\}
\end{align*}
and
\begin{align*}
K^g_{\stat} \doteq
\left\{A \in \R^{6\times 2}:
A =
\left(
\begin{array}{c}
X\\
g'(\Det(X))X\\
(g'(\Det(X))\Det(X) - g(\Det(X)))J
\end{array}
\right)
\right\}.
\end{align*}
The discussion of Subsection \ref{APPSOL} motivates Open Question \ref{open} and Kirchheim, M\"uller and \v Sver\'ak's result $(K^g)^{rc} = K^g$ for $g(x) = x^2$. If one considers approximates solutions $v_n$ to $K_{\out}^g$ or even $\KK$, one cannot expect strong convergence in $L^1$ of the gradients of $v_n$. Let us show this. Consider two matrices $A,B \in \R^{2\times 2}$ with
\[
\det(A) = \det(B) = 1, \quad \rank(A-B) = 1.
\]
Let also $C = \frac{A + B}{2}$. Notice that $\det(C) = 1$. Then, according to \cite[Theorem 3.5]{KIRK}, which is taken from \cite[Lemma 6.3]{MSCONS}, for all $n \in \N$ we can find a piecewise affine and Lipschitz map $u_n: B_1 \to \R^2$ with Lipschitz norm independent of $n$ such that
\begin{enumerate}
\item $\det(Du_n) = 1$ a.e. on $B_1 \subset \R^2$;\label{11}
\item $\|u_n(x) - Cx\|_{L^\infty(B_1,\R^2)} \le \frac{1}{n}$;\label{22}
\item $|\{x \in B_1: Du_n(x) = A\}| \ge \left(1-\frac{1}{n}\right)\frac{|B_1|}{2}$ and $|\{x \in B_1: Du_n(x) = B\}| \ge \left(1-\frac{1}{n}\right)\frac{|B_1|}{2}$.\label{33}
\end{enumerate}
By \eqref{22} and the equiboundedness of $(u_n)_n$, the sequence $(Du_n)_n$ converges weakly in $L^2$ to $C$, but it cannot converge strongly due to $\eqref{33}$. Now consider $v_n: B_1 \to \R^4$, $w_n: B_1 \to \R^6$ defined as
\[
v_n \doteq \left(\begin{array}{c} u_n\\ g'(1)u_n \end{array}\right), \quad w_n \doteq \left(\begin{array}{c} u_n\\ g'(1)u_n \\ (g'(1)-g(1))Jx \end{array}\right).
\]
By \eqref{11}, it is simple to see that $\dist(Dv_n,K^g) = \dist(Dw_n,\KK)  = 0$ a.e. in $B_1$, but  $(Dv_n)_n$ and $(Dw_n)_n$ do not converge strongly, since $(Du_n)_n$ does not.
\\
\\
In this situation, the best one can hope for is strong compactness of the sequence $(\det(Du_n))_n$, for the same reason that exact solutions to \eqref{EL} are only expected to have constant determinant, without further regularity properties. This is precisely what we are going to show for the differential inclusion defined by $\KK$. First, we establish a technical result concerning \emph{polyconvex} measures supported in $\KK$, Proposition \ref{poly}, which is inspired by \cite{SAP}. By \eqref{incconv}, this immediately yields
\[
\KK = (\KK)^{qc}.
\]
Next, we are going to use this proposition to deduce that a sequence of approximate solutions has strongly convergent jacobians in Theorem \ref{detstro}, i.e. Theorem \ref{intro:detstro} of the introduction.

\begin{prop}\label{poly}
Let $\mu \in \mathcal{P}(\R^{6\times 2})$ be a polyconvex measure, i.e.
\[
\langle\mu,\Det_{ij}(\cdot)\rangle = \Det_{ij}(\langle\mu,\id\rangle),\quad \forall 1\le i< j \le 6,
\]
supported on $K^g_{\stat}$.
Then, there exists $D \in \R$ such that
\begin{equation}\label{D}
\spt(\mu) \subset \left\{A \in \R^{6\times 2}:
A =
\left(
\begin{array}{c}
X\\
g'(\Det(X))X\\
(g'(\Det(X))\Det(X) - g(\Det(X)))J
\end{array}
\right), \Det(X) = D
\right\}.
\end{equation}
In particular,
\begin{equation}\label{bar}
M\doteq\langle\mu,\id\rangle \in \KK.
\end{equation}
\end{prop}
\begin{proof}
For a matrix $A \in \R^{6 \times 2}$ and $1\le i < j \le 6$, denote by $A^{ij}$ the $2\times 2$ submatrix obtained by considering only the $i$-th and $j$-th rows of $A$. Denote also $\det_{ij}(A) \doteq \det(A^{ij})$. Finally, let
\[
h(x) \doteq g'(x)x-g(x), \quad \forall x \in \R.
\]
We can start with the proof. The polyconvexity of $\mu$ yields, for every $i,j$,
\begin{equation}\label{first}
\int_{\R^{6\times 2}\times \R^{6\times 2}}\Det_{ij}(Y_1 - Y_2)d\mu(Y_1)\otimes d\mu(Y_2) = 0,
\end{equation}
see for instance the proof of \cite[Lemma 3]{SAP}. Since $\supp(\mu) \subset \KK$ we have for $\mu$-a.e. $Z \in \R^{6\times 2}$
\[
Z^{56} = \Det_{12}(Z)^2J.
\]
This implies that for $\mu\otimes \mu$ a.e. $(Y_1,Y_2)$,
\begin{equation}\label{second}
\Det_{56}(Y_1 - Y_2) = \Det_{56}(h(\Det_{12}(Y_1))J - h(\Det_{12}(Y_2))J) = (h(\Det_{12}(Y_1)) - h(\Det_{12}(Y_2)))^2.
\end{equation}
Combining $\eqref{first}$ and $\eqref{second}$, we obtain the existence of a number $e \in \R$ such that
\begin{equation}\label{Kg}
\spt(\mu) \subset K' \doteq \left\{A \in \R^{6\times 2}:
A =
\left(
\begin{array}{c}
X\\
g'(\Det(X))X\\
h(\Det(X))J
\end{array}
\right), h(\Det(X)) = e
\right\} \subset \KK.
\end{equation}
As in the proof of Corollary \ref{cor:polysec}, we have two cases. If $e = 0$ then, by assumption \eqref{HP} on $g$, we infer:
\[
\spt(\mu) \subset \left\{A \in \R^{6\times 2}:
A =
\left(
\begin{array}{c}
X\\
0\\
0
\end{array}
\right), \Det(X) = 0
\right\},
\]
and $\eqref{D}$-$\eqref{bar}$ readily follow. Otherwise, $e \neq 0$, which implies the existence of $e_1 < 0 < e_2$ such that
\begin{equation}\label{K'}
K' = \left\{A \in \R^{6\times 2}:
A = 
\left(
\begin{array}{c}
X\\
g'(\Det(X))X\\
h(\Det(X))J
\end{array}
\right), \Det(X) = e_1 \text{ or } \Det(X) = e_2
\right\}.
\end{equation}
Consider $M = \langle\mu,\id\rangle$, and denote with $M_i$, $1\le i\le 4$, the rows of $M$. Let
\[
I \doteq  \int_{\R^{4\times 2}}\Det_{14}(Y)d\mu(Y) =  \int_{\R^{4\times 2}}g'(\Det_{12}(Y))\Det_{12}(Y)d\mu(Y)
\]
and
\[
L \doteq \int_{\R^{4\times 2}}\Det_{34}(Y)d\mu(Y) =  \int_{\R^{4\times 2}}(g'(\Det_{12}(Y)))^2\Det_{12}(Y)d\mu(Y).
\]
By \eqref{HP}-\eqref{Kg}-\eqref{K'}, we have that the function $x\mapsto g'(x)x$ is nonnegative and zero only if $x$ is zero, and hence $I > 0$. Now, since $\spt(\mu) \subset \KK$ and $\mu$ is a polyconvex measure, we have
\begin{align}
&\langle\mu,\Det_{12}\rangle = \Det_{12}(M),\label{12}\\
&0 = \langle\mu,\Det_{13}\rangle = \Det_{13}(M),\label{13}\\
&0 = \langle\mu,\Det_{24}\rangle = \Det_{24}(M), \label{24}\\
&I  = \langle\mu,\Det_{14}\rangle = \Det_{14}(M), \label{14}\\
& -I = \langle\mu,\Det_{23}\rangle =\Det_{23}(M), \label{23}\\
&L = \langle\mu,\Det_{34}\rangle =\Det_{34}(M) \label{34}.
\end{align}
Since $I \neq 0$, from \eqref{14}-\eqref{23} we infer $M_1\neq 0$, $M_2 \neq 0$. From \eqref{13}-\eqref{24}, we thus find $\lambda, \mu \in \R$ such that
\[
M_3 = \lambda M_1 \text{ and } M_4 = \mu M_2.
\]
From \eqref{14}-\eqref{23} we further infer
\[
I = \mu \Det_{12}(M) \text{ and } I = \lambda \Det_{12}(M).
\]
As $I \neq 0$, we conclude $\lambda = \mu$, and thus by \eqref{34},
\[
L = \lambda^2\Det_{12}(M).
\]
Exploiting these relations and \eqref{K'}, we can rewrite \eqref{12}-\eqref{14}-\eqref{34} respectively as
\begin{align}
& te_1 + se_2 = \Det_{12}(M),\label{112}\\
& tg'(e_1)e_1 + sg'(e_2)e_2 = \lambda\Det_{12}(M), \label{114}\\
& t(g'(e_1))^2e_1 + s(g'(e_2))^2e_2 =\lambda^2\Det_{12}(M) \label{134},
\end{align}
if
\begin{equation}\label{ts}
t =\mu(\{A \in \R^{6\times 2}: \Det_{12}(A) = e_1\}) \text{ and } s = \mu(\{A \in \R^{6\times 2}: \Det_{12}(A) = e_2\}).
\end{equation}
Since $\mu \in \mathcal{P}(\R^{6 \times 2})$, $t,s \ge 0$ and $t + s = 1$. Let $a \doteq te_1, b \doteq se_2$ and $m \doteq \Det_{12}(M)$. Our goal is to show that $t = 0$ or $s = 0$, or equivalently that $a = 0$ or $b = 0$. To see this, we can take the square of \eqref{114} and use \eqref{134} to obtain:
\begin{align*}
(ag'(e_1) + g'(e_2)b)^2 &\overset{\eqref{114}}{=} \lambda^2m^2 \overset{\eqref{134}}{=}  am(g'(e_1))^2 + bm(g'(e_2))^2,
\end{align*}
from which it follows
\[
a^2(g'(e_1))^2 + b^2(g'(e_2))^2 + 2abg'(e_1)g'(e_2) = am(g'(e_1))^2 + bm(g'(e_2))^2,
\]
and hence
\[
a(a-m)(g'(e_1))^2 + b(b-m)(g'(e_2))^2 + 2abg'(e_1)g'(e_2) = 0.
\]
Using this notation, \eqref{112} reads as $a + b = m$. We can then use this relation in the previous equality to get:
\[
-ab(g'(e_1))^2 - ba(g'(e_2))^2 + 2abg'(e_1)g'(e_2) = 0.
\]
If, by contradiction, $a \neq 0$ and $b \neq 0$, then the latter can be simplified as
\[
0 = (g'(e_1)-g'(e_2))^2.
\]
Thus we find $g'(e_1) = g'(e_2)$ which is impossible by the fact that $e_2 < 0 < e_1$ and $g$ satisfies \eqref{HP}. We infer that $a = 0$ or $b = 0$, i.e. that $t = 0$ or $s = 0$. In particular, $m > 0$ if and only if $s = 0$ and $m < 0$ if and only if $t = 0$. In either case, by \eqref{ts} we deduce
\[
\spt(\mu) \subset \{A \in \R^{6\times 2}: \Det_{12}(A) = m\},
\]
which is precisely \eqref{D}. Using the latter and \eqref{112}-\eqref{114}-\eqref{134} we also find $M \in \KK$, that is \eqref{bar}.
\end{proof}

\begin{theorem}\label{detstro}
Let $\Omega \subset \R^2$ be open, bounded and convex. Let $u_n \in \Lip(\Omega,\R^6)$ be an equi-Lipschitz sequence with 
\begin{equation}\label{sconv}
\dist(Du_n(x),\KK) \to 0 \text{ in } L^1_{\loc}(\Omega).
\end{equation}
Suppose moreover that $u_n$ converges weakly-$*$ in $W^{1,\infty}$ to $u \in \Lip(\Omega,\R^6)$. Then, $\Det_{12}(Du_n)$ converges strongly in $L_{\loc}^1$ to $\Det_{12}(Du)$ and $Du \in K$.
\end{theorem}
\begin{proof}
We can equivalently show that every subsequence of $(u_n)_n$ admits a further subsequence whose jacobians strongly converge to $x \mapsto \Det(Du)(x)$. Thus, we can assume, up to passing to a non-relabeled subsequence, that $(Du_n)_n$ generates the gradient Young measure $\nu=(\nu_x)_{x \in \Omega}$. Since determinants of minors of $(Du_n)_n$ are weakly convergent, see \cite[Theorem 2.3]{DMU}, we have that $\nu_x$ is a polyconvex probability measure for a.e. $x$. Indeed for all $\varphi \in C^\infty_c(\Omega)$, and $1\le i< j \le 6$, using the notation of the previous proof, we have:
\[
\int_{\Omega}\varphi(x)\int_{\R^{6\times 2}}\Det_{ij}(A)d\nu_x(A)dx = \lim_n\int_{\Omega}\varphi(x)\Det_{ij}(Du_n)(x)dx = \int_{\Omega}\varphi(x)\Det_{ij}(Du)(x)dx,
\]
from which we infer, for a.e. $x \in \Omega$
\begin{equation}\label{nulllag}
\int_{\R^{6\times 2}}\Det_{ij}(A)d\nu_x(A) = \Det_{ij}(Du)(x) \overset{\eqref{exp}}{=} \Det_{ij}\left(\int_{\R^{6\times 2}}Ad\nu_x(A)\right),
\end{equation}
which precisely shows that $\nu_x$ is a polyconvex measure for a.e. $x \in \Omega$. Assumption \eqref{sconv} implies that $\nu_x$ is supported in $\KK$ for a.e. $x \in \Omega$. We can therefore apply the previous Proposition to obtain the existence of a number $D = D(x)$ such that
\[
\spt(\nu_x) \subset \left\{A \in \R^{6\times 2}:
A =
\left(
\begin{array}{c}
X\\
g'(\Det(X))X\\
(g'(\Det(X))\Det(X) - g(\Det(X)))J
\end{array}
\right), \Det(X) = D(x)
\right\}.
\]
Finally, \eqref{nulllag} implies $D(x) = \Det_{12}(Du(x))$ for a.e. $x$. In particular, the Young measure generated by the sequence $(\Det_{12}(Du_n))_n$ is a Dirac's delta at $\Det_{12}(Du)$ at a.e. $x \in \Omega$. By Corollary \ref{strongc}, this implies the a.e. convergence of the sequence $(\Det_{12}(Du_n))_n$. The fact that $Du \in \KK$ is a consequence of \eqref{bar}.
\end{proof}

\bibliographystyle{plain}
\bibliography{Detsquare}
\end{document}